\documentclass{article}

\usepackage{amsmath,amssymb,amsthm}
\usepackage{a4wide}
\usepackage{cite}

\newtheorem{Theorem}{Theorem}
\newtheorem{Cor}[Theorem]{Corollary}
\newtheorem{Lemma}[Theorem]{Lemma}
\newtheorem{Proposition}[Theorem]{Proposition}
\newtheorem{Ex}[Theorem]{Example}
\theoremstyle{Remark}
\newtheorem{Remark}[Theorem]{Remark}
\theoremstyle{Definition}
\newtheorem{Definition}[Theorem]{Definition}

\begin{document}

\title{Generalizing the variational theory on time scales\\
to include the delta indefinite integral\thanks{Submitted 25-Jul-2010;
revised 27-Nov-2010; accepted 16-Feb-2011;
for publication in \emph{Computers \& Mathematics with Applications}.}}

\author{Nat\'{a}lia Martins\\
\texttt{natalia@ua.pt}
\and Delfim F. M. Torres\thanks{Corresponding author.}\\
\texttt{delfim@ua.pt}}

\date{Department of Mathematics\\
University of Aveiro\\
3810-193 Aveiro, Portugal}

%--------------------------------------------------

\maketitle

\begin{abstract}
We prove necessary optimality conditions of Euler--Lagrange type
for generalized problems of the calculus of variations
on time scales with a Lagrangian depending not only
on the independent variable, an unknown function and its delta derivative,
but also on a delta indefinite integral that depends on the unknown function.
Such kind of variational problems were considered by Euler himself
and have been recently investigated in
[Methods Appl. Anal. 15 (2008), no.~4, 427--435].
Our results not only provide a generalization to previous results,
but also give some other interesting optimality conditions as special cases.

\bigskip


\bigskip

\end{abstract}

%--------------------------------------------------

\section{Introduction}

In what follows, $\mathbb{T}$ denotes a time scale
with operators $\sigma$, $\rho$,
$\mu$, $\nu$, $\Delta$, and $\nabla$
\cite{Bohner:Peterson,Bohner-Peterson2}.
We also assume that there exist at least three points
on the time scale: $a,b,s \in \mathbb{T}$ with $a<b<s$, and
that the operator $\sigma$ is delta differentiable.
The main purpose of this paper is to generalize
the Calculus of Variations  on time scales
(see \cite{Bohner,EJC:ZND,MyID:140,MyID:175,MyID:170,MyID:183}
and references therein) by considering the variational problem
\begin{equation}
\label{functional}
\mathcal{L}\left(y\right)
= \int_{a}^{b} L\left(t,y^{\sigma}(t), y^\Delta(t), z(t) \right) \Delta t
\longrightarrow \mathrm{extr},
\end{equation}
where ``extr'' denotes ``extremize'' (\textrm{i.e.}, minimize or maximize)
and the variable $z$ in the integrand is itself expressed
in terms of an indefinite integral
$$
z(t)=\int_{a}^{t} g\left(
\tau,y^{\sigma}(\tau), y^\Delta(\tau) \right) \Delta \tau.
$$
In Subsection~\ref{subsection E-L} we obtain the Euler--Lagrange
equation for problem (\ref{functional}) in the class of functions
$y \in C^1_{rd}(\mathbb{T},\mathbb{R})$ satisfying the boundary conditions
\begin{equation}
\label{boundary conditions}
y(a)=\alpha \quad \mbox{ and } \quad y(b)=\beta
\end{equation}
for some fixed $\alpha, \beta \in \mathbb{R}$
(\textrm{cf.} Theorem~\ref{E-L}).
Accordingly to Fraser \cite{Fraser}, the ideia of generalizing
the basic problem of the calculus of variations by considering
a variational integral depending also on an indefinite integral
(in the classical setting, that is, when $\mathbb{T}=\mathbb{R}$)
was first considered by Euler in 1741. Our Euler--Lagrange equation
is a generalization of the Euler--Lagrange equations obtained by Euler
\cite[Eq. (8)]{Fraser}, Bohner \cite{Bohner}, and Gregory \cite{J:Gregory}.
The transversality conditions for problem (\ref{functional})
are obtained in Subsection~\ref{subsection natural boundary}.
In Subsection~\ref{subsectin isoperimetric} we prove a necessary
optimality condition for the isoperimetric problem: problem
(\ref{functional})--(\ref{boundary conditions})
subject to the delta integral constraint
$$
\mathcal{J}\left(y\right)
= \int_{a}^{b} F\left(t,y^{\sigma}(t), y^\Delta(t), z(t) \right) \Delta t
= \gamma
$$
for some given $\gamma \in \mathbb{R}$.
In Subsection~\ref{duality} we explain how it is possible to prove
backward versions of our results by means of Caputo's duality \cite{Caputo}
(see also \cite{comBasia:duality}). Finally, in  Section~\ref{subsection applications}
we provide some applications of our main results.

%--------------------------------------------------

\section{Preliminaries}

For definitions, notations and results concerning the theory of time scales
we refer the readers to the comprehensive books \cite{Bohner:Peterson,Bohner-Peterson2}.
All the intervals in this paper are time scale intervals. Throughout the text we denote
by $\partial_i f$ the partial derivative of a function $f$ with respect to its $i$th argument.

We assume that
\begin{enumerate}
\item the admissible functions $y$ belong to the class $C^{1}_{rd}(\mathbb{T},\mathbb{R})$;
\item $(t,y,v,z) \rightarrow L(t,y,v,z)$ \ and  \ $(t,y,v,z) \rightarrow F(t,y,v,z)$
have continuous partial derivatives with respect to $y,v,z$ for all $t \in [a,b]$;
\item $(t,y,v) \rightarrow g(t,y,v)$ has continuous partial
derivatives with respect to $y,v$ for all $t \in [a,b]$;
\item $t \rightarrow  L(t, y^\sigma(t), y^\Delta(t),z(t))$ and
$t \rightarrow F(t, y^\sigma(t), y^\Delta(t),z(t))$ belong to the class
$C_{rd}(\mathbb{T},\mathbb{R})$ for any admissible function $y$;
\item $t \rightarrow  \partial_3 L(t, y^\sigma(t), y^\Delta(t),z(t))$,
$t \rightarrow  \partial_3 F(t, y^\sigma(t), y^\Delta(t),z(t))$
and $t \rightarrow \partial_3 g(t, y^\sigma(t), y^\Delta(t))$
belong to the class $C^{1}_{rd}(\mathbb{T},\mathbb{R})$
for any admissible function $y$.
\end{enumerate}

\begin{Definition}
\label{definicao de extremizer}
An admissible function $y_\ast \in  C^{1}_{rd}\left(\mathbb{T}, \mathbb{R}\right)$
is said to be a local minimizer (resp. local maximizer) to  problem
(\ref{functional})--(\ref{boundary conditions}) if there exists $\delta >0$ such that
$\mathcal{L}(y_\ast) \leq \mathcal{L}(y)$ (resp. $\mathcal{L}(y_\ast) \geq \mathcal{L}(y)$)
for all admissible $y$ satisfying the boundary conditions (\ref{boundary conditions})
and $\parallel y - y_\ast \parallel < \delta$, where
$$
\parallel y \parallel = \sup_{t \in [a,b]^\kappa}\mid y^\sigma(t)\mid
+ \sup_{t \in [a,b]^\kappa}\mid y^\Delta(t)\mid.
$$
\end{Definition}

\begin{Definition}
We say that $\eta\in C^{1}_{rd}\left(\mathbb{T},\mathbb{R}\right)$
is an admissible variation to problem
(\ref{functional})--(\ref{boundary conditions})
provided $\eta\left(  a\right)  =\eta\left(  b\right)=0$.
\end{Definition}

The following result, known as the \emph{fundamental lemma
of the calculus of variations on time scales},
is an important tool in the proofs of our main results.
The proof of Lemma~\ref{Fund. Lemma}
follows immediately from \cite[Theorem~15]{MartinsTorres}
and the duality arguments of Caputo \cite{Caputo}.

\begin{Lemma}
\label{Fund. Lemma}
Let $f \in C_{rd}([a,b],\mathbb{R})$. Then
$$
\int_a^b f(t) \eta^\sigma(t)\Delta t =0 \quad
\mbox{ for all }  \quad \eta \in C_{rd}([a,b],\mathbb{R})
\quad \mbox{ with } \quad  \eta(a)=\eta(b)=0
$$
if and only if $f(t)=0$ for all $t \in [a,b]^\kappa$.
\end{Lemma}

%--------------------------------------------------

\section{Main results}

In order to simplify expressions, we  introduce two operators,
$[\cdot]$ and $\{\cdot\}$, defined in the following way:
$$
[y](t) :=(t, y^\sigma(t), y^\Delta(t), z(t))
\quad \mbox{and} \quad
\{y\}(t) :=(t, y^\sigma(t), y^\Delta(t)),
$$
where $y \in C^{1}_{rd}(\mathbb{T},\mathbb{R})$.

%--------------------------------------------------

\subsection{Euler--Lagrange equation}
\label{subsection E-L}

\begin{Theorem}[Necessary optimality condition
to (\ref{functional})--(\ref{boundary conditions})]
\label{E-L}
Suppose that $y_{\ast}$ is  a local minimizer
or local maximizer to problem
(\ref{functional})--(\ref{boundary conditions}).
Then $y_{\ast}$ satisfies the Euler--Lagrange equation
\begin{equation}
\label{E-L equation}\displaystyle \partial_{2}L [y](t)
- \frac{\Delta}{\Delta t}\partial_{3} L [y](t)
+ \displaystyle \partial_2 g\{y\}(t)
\cdot \int_{\sigma(t)}^b \partial_4 L [y](\tau)\Delta\tau
- \displaystyle \frac{\Delta}{\Delta t}\left(\partial_3 g\{y\}(t)
\cdot \int_{\sigma(t)}^b \partial_4 L [y](\tau)\Delta\tau\right)=0
\end{equation}
for all $t\in\left[a,b\right]^\kappa$.
\end{Theorem}

\begin{proof}
Suppose that $y_{\ast}$ is a local minimizer (resp. maximizer) to problem
(\ref{functional})--(\ref{boundary conditions}).
Let $\eta$ be an admissible variation and define the function
$\phi:\mathbb{R}  \rightarrow \mathbb{R}$
by $\phi\left(  \epsilon\right)  :=\mathcal{L}(y_{\ast}+\epsilon \eta)$.
It is clear that a necessary condition for $y_{\ast}$
to be an extremizer is given by $\phi^{\prime}\left(  0\right)  =0$.
Note that
\begin{multline*}
\phi^\prime (0)
=  \displaystyle \int_a^b \Bigl( \partial_{2}L[y_\ast](t)\eta^\sigma(t)
+ \partial_{3}L [y_\ast](t) \eta^\Delta(t) \\
+  \displaystyle \partial_{4}L [y_\ast](t)
\cdot \int_a^t \left(\partial_2 g\{y_\ast\}(\tau) \eta^\sigma(\tau)
+ \partial_3 g\{y_\ast\}(\tau) \eta^\Delta(\tau)\right) \Delta \tau \Bigr) \Delta t.
\end{multline*}
Using the integration by parts formula, we obtain
$$
\displaystyle \int_a^b \partial_{3}L [y_\ast](t) \eta^\Delta(t) \Delta t
= \Bigl[\partial_{3}L [y_\ast](t)\eta(t)\Bigr]_a^b
- \int_a^b \frac{\Delta}{\Delta t}\partial_{3}L [y_\ast](t)\eta^\sigma(t) \Delta t,
$$
\begin{equation*}
\begin{split}
\int_a^b & \left(\partial_{4}L [y_\ast](t) \cdot
\int_a^t \partial_2 g\{y_\ast\}(\tau) \eta^\sigma(\tau) \Delta \tau \right)\Delta t  \\
&=   \displaystyle \left[\int_b^t \partial_{4}L [y_\ast](\tau) \Delta \tau
\cdot \int_a^t \partial_2 g\{y_\ast\}(\tau) \eta^\sigma(\tau) \Delta \tau \right ]_a^b
- \displaystyle \int_a^b \left( \int_b^{\sigma(t)} \partial_{4}L[y_\ast](\tau) \Delta \tau
\cdot \partial_2 g\{y_\ast\}(t) \eta^\sigma(t)\right)\Delta t\\
&= -
\displaystyle \int_a^b \left( \partial_2 g\{y_\ast\}(t) \cdot
\int_b^{\sigma(t)} \partial_{4}L [y_\ast](\tau) \Delta \tau \right) \eta^\sigma(t) \Delta t,
\end{split}
\end{equation*}
and
\begin{equation*}
\begin{split}
\int_a^b &\left(\partial_{4}L [y_\ast](t) \cdot
\int_a^t \partial_3 g\{y_\ast\}(\tau) \eta^\Delta(\tau) \Delta \tau \right)\Delta t \\
&= \displaystyle \left[\int_b^t \partial_{4}L [y_\ast](\tau) \Delta \tau \cdot
\int_a^t \partial_3 g\{y_\ast\}(\tau) \eta^\Delta(\tau) \Delta \tau \right ]_a^b
- \displaystyle \int_a^b \left( \int_b^{\sigma(t)} \partial_{4}L[y_\ast](\tau) \Delta \tau
\cdot \partial_3 g\{y_\ast\}(t) \eta^\Delta(t)\right) \Delta t\\
&= - \displaystyle \int_a^b \left( \partial_3 g\{y_\ast\}(t) \cdot
\int_b^{\sigma(t)} \partial_{4}L [y_\ast](\tau) \Delta \tau \right) \eta^\Delta(t) \Delta t.
\end{split}
\end{equation*}
Using again integration by parts in the last integral we obtain
\begin{multline*}
\int_a^b \left( \partial_3 g\{y_\ast\}(t) \cdot
\int_b^{\sigma(t)} \partial_{4}L [y_\ast](\tau) \Delta \tau \right) \eta^\Delta(t) \Delta t\\
=\left[  \partial_3 g\{y_\ast\}(t) \cdot \int_b^{\sigma(t)}
\partial_{4}L [y_\ast](\tau) \Delta \tau  \cdot \eta(t) \right ]_a^b
- \displaystyle \int_a^b \frac{\Delta}{\Delta t}\left( \partial_3 g\{y_\ast\}(t)
\cdot \int_b^{\sigma(t)} \partial_{4}L [y_\ast](\tau) \Delta \tau \right) \eta^\sigma(t) \Delta t.
\end{multline*}
Since $\eta(a)=\eta(b)=0$, then
\begin{multline*}
\phi^\prime(0) = \int_a^b \Bigl( \partial_{2}L[y_\ast](t)
- \frac{\Delta}{\Delta t}\partial_{3}L[y_\ast](t)
- \displaystyle \partial_2 g\{y_\ast\}(t) \cdot
\int_b^{\sigma(t)} \partial_{4}L [y_\ast](\tau) \Delta \tau \\
+ \frac{\Delta}{\Delta t}( \partial_3 g\{y_\ast\}(t)
\cdot \int_b^{\sigma(t)} \partial_{4}L [y_\ast](\tau) \Delta \tau )\Bigr) \eta^\sigma (t) \Delta t.
\end{multline*}
From the optimality condition $\phi^\prime(0) =0$ we conclude,
by the fundamental lemma of the calculus of variations on time scales
(Lemma~\ref{Fund. Lemma}), that
$$
\displaystyle \partial_{2}L[y_\ast](t)
- \frac{\Delta}{\Delta t}\partial_{3}L[y_\ast](t)
- \displaystyle \partial_2 g\{y_\ast\}(t) \cdot
\int_b^{\sigma(t)} \partial_{4}L [y_\ast](\tau) \Delta \tau
+ \displaystyle \frac{\Delta}{\Delta t}\left( \partial_3 g\{y_\ast\}(t)
\cdot \int_b^{\sigma(t)} \partial_{4}L [y_\ast](\tau) \Delta \tau \right) =0
$$
for all $t\in\left[  a,b\right]^\kappa$, proving the desired result.
\end{proof}

\begin{Remark}
Note that
\begin{enumerate}
\item The Euler--Lagrange equation  (\ref{E-L equation})
is a generalization of the Euler--Lagrange equation obtained
by Euler in 1741 (if $\mathbb{T}=\mathbb{R}$,
we obtain equation (8) of \cite{Fraser}).

\item Theorem 3.1 of \cite{J:Gregory} is a corollary
of Theorem~\ref{E-L}:  choose $g(t,u,v)=u$
and consider the time scale to be the set of real numbers.

\item The Euler--Lagrange equation for the basic problem
of the Calculus of Variations on time scales
(see, \textrm{e.g.}, \cite{Bohner}) is easily obtained
from Theorem~\ref{E-L}: in this case, $\partial_4 L=0$
and therefore we get the equation
$$
\partial_{2}L(t,y^{\sigma}(t),y^\Delta(t))
- \frac{\Delta}{\Delta t}\partial_{3}L(t,y^{\sigma}(t),y^\Delta(t))=0
$$
for all $t\in\left[  a,b\right]^\kappa$.
\end{enumerate}
\end{Remark}

\begin{Remark} Theorem \ref{E-L} gives the Euler--Lagrange equation in the
delta-differential form. As in the classical case, one can obtain the
Euler--Lagrange equation in the integral form. More precisely, the
Euler--Lagrange equation in the delta-integral form to
problem (\ref{functional})--(\ref{boundary conditions}) is
$$
\displaystyle \partial_{3}L [y](t)
+ \partial_3 g\{y\}(t)
\cdot \int_{\sigma(t)}^b \partial_4 L [y](\tau)\Delta\tau
+ \int_t^b \Big(\partial_2 L [y](s)+
\displaystyle \partial_2 g\{y\}(s)
\cdot \int_{\sigma(s)}^b \partial_{4} L [y](\tau) \Delta\tau\Big)\Delta s=const.
$$
\end{Remark}

%--------------------------------------------------

\subsection{Natural boundary conditions}
\label{subsection natural boundary}

We now consider the case when the values $y(a)$ and $y(b)$
are not necessarily specified.

\begin{Theorem}[Natural boundary conditions to (\ref{functional})]
\label{Theorem natural boundary conditions}
Suppose that $y_{\ast}$ is  a local minimizer (resp. local maximizer) to problem
(\ref{functional}).
Then $y_{\ast}$ satisfies
the Euler--Lagrange equation
(\ref{E-L equation}). Moreover,
\begin{enumerate}

\item  if  $y(a)$ is free,  then the natural boundary condition
\begin{equation}
\label{a}
\partial_3 L[y_\ast](a)= - \partial_3 g\{y_\ast\}(a)
 \cdot \int_{\sigma(a)}^b \partial_{4}L[y_\ast](\tau)\Delta\tau
\end{equation}
holds;
\item if $y(b)$ is free, then the natural boundary condition
\begin{equation}
\label{b}
\partial_3 L[y_\ast](b)=\partial_3 g\{y_\ast\}(b)\cdot
\int^{\sigma(b)}_b \partial_{4}L[y_\ast](\tau)\Delta\tau
\end{equation}
holds.
\end{enumerate}
\end{Theorem}

\begin{proof}
Suppose that $y_{\ast}$ is a local minimizer
(resp. maximizer) to problem (\ref{functional}).
Let $\eta \in C^1_{rd}(\mathbb{T},\mathbb{R})$
and define the function $\phi:\mathbb{R}  \rightarrow \mathbb{R}$
by $\phi\left(  \epsilon\right)  :=\mathcal{L}( y_{\ast}+\epsilon \eta)$.
It is clear that a necessary condition for $y_{\ast}$ to be an extremizer
is given by $\phi^{\prime}\left(  0\right)=0$.
From the arbitrariness of $\eta$, and using similar arguments
as the ones used in the proof of Theorem~\ref{E-L}, we conclude that
$y_{\ast}$ satisfies the Euler--Lagrange equation (\ref{E-L equation}).

\begin{enumerate}
\item Suppose now that $y(a)$ is free.
If $y(b)=\beta$ is given, then $\eta(b)=0$;
if $y(b)$ is free, then we restrict ourselves
to those $\eta$ for which $\eta(b)=0$.
Therefore,
\begin{equation}
\label{boundary condition a}
\begin{split}
0 &=  \phi^\prime(0)\\
&= \displaystyle \int_a^b \Bigl( \partial_{2}L[y_\ast](t)
- \frac{\Delta}{\Delta t}\partial_{3}L[y_\ast](t)
- \displaystyle \partial_2 g\{y_\ast\}(t) \cdot
\int_b^{\sigma(t)} \partial_{4}L [y_\ast](\tau) \Delta \tau \\
& \quad + \displaystyle \frac{\Delta}{\Delta t}( \partial_3 g\{y_\ast\}(t)
\cdot \int_b^{\sigma(t)} \partial_{4}L [y_\ast](\tau) \Delta \tau )\Bigr) \eta^\sigma (t) \Delta t\\
& \quad -  \displaystyle
\partial_{3}L[y_\ast](a) \cdot \eta(a) + \partial_3 g\{y_\ast\}(a)
\cdot \int_b^{\sigma(a)} \partial_{4}L[y_\ast](\tau)\Delta\tau \cdot \eta(a).
\end{split}
\end{equation}
Using the Euler--Lagrange equation (\ref{E-L equation})
into (\ref{boundary condition a}) we obtain
$$
\left(- \partial_{3}L[y_\ast](a) + \partial_3 g\{y_\ast\}(a)
\cdot \int_b^{\sigma(a)} \partial_{4}L[y_\ast](\tau) \Delta \tau\right) \cdot \eta(a)=0.
$$
From the arbitrariness of $\eta$ it follows that
$$
\partial_{3}L[y_\ast](a)= \partial_3 g\{y_\ast\}(a)
 \cdot \int_b^{\sigma(a)} \partial_{4}L[y_\ast](\tau)\Delta\tau.
$$

\item Suppose now that $y(b)$ is free.
If $y(a)=\alpha$, then $\eta(a)=0$; if $y(a)$ is
 free, then we restrict ourselves to those $\eta$ for which $\eta(a)=0$.
Thus,
\begin{equation}
\label{boundary condition b}
\begin{split}
0 &= \phi^\prime(0)\\
&=\displaystyle \int_a^b \Bigl( \partial_{2}L[y_\ast](t)
- \frac{\Delta}{\Delta t}\partial_{3}L[y_\ast](t)
- \displaystyle \partial_2 g\{y_\ast\}(t) \cdot
\int_b^{\sigma(t)} \partial_{4}L [y_\ast](\tau) \Delta \tau \\
& \qquad + \displaystyle \frac{\Delta}{\Delta t}( \partial_3 g\{y_\ast\}(t)
\cdot \int_b^{\sigma(t)} \partial_{4}L [y_\ast](\tau)
\Delta \tau )\Bigr) \eta^\sigma (t) \Delta t\\
&\qquad + \partial_{3}L[y_\ast](b) \cdot \eta(b)-
\partial_3 g\{y_\ast\}(b)\cdot \int^{\sigma(b)}_b \partial_{4}L[y_\ast](\tau)\Delta\tau
.
\end{split}
\end{equation}
Using the Euler--Lagrange equation (\ref{E-L equation})
into (\ref{boundary condition b}),
and from the arbitrariness of $\eta$, it follows that
$$\partial_{3}L[y_\ast](b)=
\partial_3 g\{y_\ast\}(b)\cdot \int^{\sigma(b)}_b \partial_{4}L[y_\ast](\tau)\Delta\tau.$$
\end{enumerate}
\end{proof}

\begin{Remark}
In the classical setting, $\mathbb{T}=\mathbb{R}$
and $L$ does not depend on $z$. Then, equations (\ref{a})
and (\ref{b}) reduce to the well-known natural boundary conditions
$$
\partial_{3}L(a,y_\ast(a),y^{\prime}_\ast(a))=0
\quad \mbox{ and } \quad \partial_{3}L(b,y_\ast(b),y^{\prime}_\ast(b))=0,
$$
respectively.
\end{Remark}

%--------------------------------------------------

\subsection{Isoperimetric problem}
\label{subsectin isoperimetric}

We now study the isoperimetric problem on time scales
with a delta integral constraint, both for normal
and abnormal extremizers. The problem consists
of minimizing or maximizing the functional
\begin{equation}
\label{functional 1}
\mathcal{L}(y) =  \int_{a}^{b} L\left(
t,y^{\sigma}(t), y^\Delta(t), z(t) \right) \Delta t,
\end{equation}
where the variable $z$ in the integrand is itself
expressed in terms of an indefinite delta integral
$$
z(t)=\int_{a}^{t} g\left(
\tau,y^{\sigma}(\tau), y^\Delta(\tau) \right) \Delta \tau,
$$
in the class of functions
$y \in C_{rd}^1(\mathbb{T},\mathbb R)$,
satisfying the boundary conditions
\begin{equation}
\label{boundary conditions 1}
y(a)=\alpha \quad \mbox{ and } \quad y(b)=\beta
\end{equation}
and the delta integral constraint
\begin{equation}
\label{isoperimetric 1}
\mathcal{J}(y)  =  \int_{a}^{b} F\left(
t,y^{\sigma}(t), y^\Delta(t), z(t) \right) \Delta t  =\gamma
\end{equation}
for some given $\alpha, \beta, \gamma \in \mathbb{R}$.

\begin{Definition}
We say that $y_\ast \in C_{rd}^{1}(\mathbb{T},\mathbb{R})$
is a  local minimizer (resp. local maximizer) to the isoperimetric problem
(\ref{functional 1})--(\ref{isoperimetric 1})
if there exists $\delta >0$ such that $\mathcal{L}(y_\ast) \leq \mathcal{L}(y)$
(resp. $\mathcal{L}(y_\ast) \geq \mathcal{L}(y)$) for all admissible $y$ satisfying
the boundary conditions (\ref{boundary conditions 1}), the isoperimetric constraint
(\ref{isoperimetric 1}), and $\parallel y - y_\ast \parallel < \delta$.
\end{Definition}

\begin{Definition}
We say that $y \in C_{rd}^1(\mathbb{T},\mathbb R)$ is an extremal
to $\mathcal{J}$ if $y$ satisfies the Euler--Lagrange equation
(\ref{E-L equation}) relatively to $\mathcal{J}$.
An extremizer (\textrm{i.e.}, a local minimizer or a local maximizer) to problem
(\ref{functional 1})--(\ref{isoperimetric 1})
that is not an extremal to $\mathcal{J}$ is said to be a normal extremizer; otherwise
(\textrm{i.e.}, if it is an extremal to $\mathcal{J}$), the extremizer is said to be abnormal.
\end{Definition}

\begin{Theorem}[Necessary optimality condition for normal extremizers
of (\ref{functional 1})--(\ref{isoperimetric 1})]
\label{normalcase}
Suppose that $y_\ast \in C_{rd}^1(\mathbb{T},\mathbb R)$
gives a local minimum or a local maximum
to the functional $\mathcal{L}$
subject to the boundary conditions (\ref{boundary conditions 1})
and the integral constraint (\ref{isoperimetric 1}). If $y_\ast$
is not an extremal to $\mathcal{J}$, then there exists
a real $\lambda$ such that $y_\ast$ satisfies the equation
\begin{equation}
\label{E-L equation for H}
\displaystyle \partial_{2}H  [y](t)
- \frac{\Delta}{\Delta t}\partial_{3} H [y](t)
+\displaystyle \partial_2 g\{y\}(t) \cdot \int_{\sigma(t)}^b \partial_4 H [y](\tau)\Delta\tau
- \displaystyle \frac{\Delta}{\Delta t}\left(\partial_3 g\{y\}(t)
\cdot \int_{\sigma(t)}^b \partial_4 H [y](\tau) \Delta\tau\right)=0
\end{equation}
for all $t \in [a,b]^\kappa$, where $H = L-\lambda F$.
\end{Theorem}

\begin{proof}
Suppose that $y_\ast \in C_{rd}^1(\mathbb{T},\mathbb R)$ is a normal
extremizer to problem (\ref{functional 1})--(\ref{isoperimetric 1}).
Define the real functions $\phi, \psi:\mathbb{R}^2 \rightarrow \mathbb{R}$ by
\begin{gather*}
\phi(\epsilon_1, \epsilon_2)
= \mathcal{I}(y_\ast + \epsilon_1 \eta_1 + \epsilon_2 \eta_2),\\
\psi(\epsilon_1, \epsilon_2)
= \mathcal{J}(y_\ast + \epsilon_1 \eta_1 + \epsilon_2 \eta_2)  - \gamma,
\end{gather*}
where $\eta_2$ is a fixed variation (that we will choose later) and $\eta_1$ is an arbitrary variation.
Note that
\begin{multline*}
\frac{\partial \psi}{\partial \epsilon_2}(0,0)
= \displaystyle\int_a^b \Bigl( \partial_{2}F[y_\ast](t)\eta_2^\sigma(t)
+ \partial_{3}F [y_\ast](t) \eta_2^\Delta(t)\\
+ \partial_{4}F [y_\ast](t)\cdot
\int_a^t \left(\partial_2 g\{y_\ast\}(\tau) \eta_2^\sigma(\tau)
+ \partial_3 g\{y_\ast\}(\tau) \eta_2^\Delta(\tau)\right) \Delta \tau \Bigr) \Delta t.
\end{multline*}
Integration by parts and $\eta_2(a)=\eta_2(b)=0$ gives
\begin{multline*}
\frac{\partial \psi}{\partial \epsilon_2}(0,0)
= \displaystyle \int_a^b \Bigl( \partial_{2}F[y_\ast](t)
- \frac{\Delta}{\Delta t}\partial_{3}F[y_\ast](t)
- \displaystyle \partial_2 g\{y_\ast\}(t) \cdot
\int_b^{\sigma(t)} \partial_{4}F [y_\ast](\tau) \Delta \tau \\
+ \displaystyle \frac{\Delta}{\Delta t}( \partial_3 g\{y_\ast\}(t)
\cdot \int_b^{\sigma(t)} \partial_{4}F [y_\ast](\tau) \Delta \tau )\Bigr)
\eta_2^\sigma (t) \Delta t.
\end{multline*}
Since, by hypothesis, $y_\ast$ is not an extremal to $\mathcal{J}$,
then we can choose $\eta_2$ such that
$\displaystyle \frac{\partial \psi}{\partial \epsilon_2}(0,0) \neq 0$.
We keep $\eta_2$ fixed. Since $\psi(0,0)=0$, by the implicit function theorem
there exists a function $h$, defined in a neighborhood $V$ of zero,
such that $h(0)=0$ and $\psi(\epsilon_1,h(\epsilon_1))=0$
for any $\epsilon_1 \in V$, \textrm{i.e.}, there exists a subset of variation curves
$y= y_\ast + \epsilon_1 \eta_1 + h(\epsilon_1)\eta_2$
satisfying the isoperimetric constraint.
Note that  $(0,0)$ is an extremizer of $\phi$ subject to the constraint $\psi=0$ and
$$
\nabla \psi (0,0) \neq (0,0).
$$
By the Lagrange multiplier rule (\textrm{cf.}, \textrm{e.g.}, \cite{vanBrunt}),
there exists some constant $\lambda \in \mathbb{R}$ such that
\begin{equation}\label{gradient} \nabla \phi (0,0) = \lambda \nabla \psi (0,0).
\end{equation}
Since
\begin{multline*}
\frac{\partial \phi}{\partial \epsilon_1}(0,0)
= \displaystyle \int_a^b \Bigl( \partial_{2}L[y_\ast](t)
- \frac{\Delta}{\Delta t}\partial_{3}L[y_\ast](t)
- \displaystyle \partial_2 g\{y_\ast\}(t) \cdot
\int_b^{\sigma(t)} \partial_{4}L [y_\ast](\tau) \Delta \tau \\
+ \displaystyle \frac{\Delta}{\Delta t}( \partial_3 g\{y_\ast\}(t)
\cdot \int_b^{\sigma(t)} \partial_{4}
L [y_\ast](\tau) \Delta \tau )\Bigr) \eta_1^\sigma (t) \Delta t
\end{multline*}
and
\begin{multline*}
\frac{\partial \psi}{\partial \epsilon_1}(0,0)
= \displaystyle \int_a^b \Bigl( \partial_{2}F[y_\ast](t)
- \frac{\Delta}{\Delta t}\partial_{3}F[y_\ast](t)
- \displaystyle \partial_2 g\{y_\ast\}(t) \cdot
\int_b^{\sigma(t)} \partial_{4}F [y_\ast](\tau) \Delta \tau \\
+ \frac{\Delta}{\Delta t}( \partial_3 g\{y_\ast\}(t)
\cdot \int_b^{\sigma(t)} \partial_{4}F [y_\ast](\tau) \Delta \tau )\Bigr)
\eta_1^\sigma (t) \Delta t,
\end{multline*}
it follows from (\ref{gradient}) that
\begin{equation*}
\begin{split}
0 &=\displaystyle \int_a^b \Bigl( \partial_{2}L[y_\ast](t)
- \frac{\Delta}{\Delta t}\partial_{3}L[y_\ast](t)-
\displaystyle \partial_2 g\{y_\ast\}(t) \cdot
\int_b^{\sigma(t)} \partial_{4}L [y_\ast](\tau) \Delta \tau \\
&\qquad + \displaystyle \frac{\Delta}{\Delta t}( \partial_3 g\{y_\ast\}(t)
\cdot \int_b^{\sigma(t)} \partial_{4}L [y_\ast](\tau) \Delta \tau )\\
&\qquad -  \displaystyle \lambda \Bigl(\partial_{2}F[y_\ast](t)
- \frac{\Delta}{\Delta t}\partial_{3}F[y_\ast](t)-
\displaystyle \partial_2 g\{y_\ast\}(t) \cdot
\int_b^{\sigma(t)} \partial_{4}F [y_\ast](\tau) \Delta \tau \\
&\qquad + \displaystyle \frac{\Delta}{\Delta t}( \partial_3 g\{y_\ast\}(t)
\cdot \int_b^{\sigma(t)} \partial_{4}F [y_\ast](\tau) \Delta \tau )\Bigr)\Bigr)
\eta_1^\sigma (t) \Delta t.
\end{split}
\end{equation*}
Using the fundamental lemma of the calculus of variations
(Lemma~\ref{Fund. Lemma}), and recalling that $\eta_1$
is arbitrary, we conclude that
\begin{equation*}
\begin{split}
0&= \displaystyle\partial_{2}L[y_\ast](t)
- \frac{\Delta}{\Delta t}\partial_{3}L[y_\ast](t)
-\displaystyle \partial_2 g\{y_\ast\}(t)
\cdot \int_b^{\sigma(t)} \partial_{4}L [y_\ast](\tau) \Delta \tau \\
&\qquad + \displaystyle \frac{\Delta}{\Delta t}( \partial_3 g\{y_\ast\}(t)
\cdot \int_b^{\sigma(t)} \partial_{4}L [y_\ast](\tau) \Delta \tau )\\
&\qquad - \displaystyle \lambda \Bigl(\partial_{2}F[y_\ast](t)
- \frac{\Delta}{\Delta t}\partial_{3}F[y_\ast](t)
- \displaystyle \partial_2 g\{y_\ast\}(t) \cdot
\int_b^{\sigma(t)} \partial_{4}F [y_\ast](\tau) \Delta \tau \\
&\qquad + \displaystyle \frac{\Delta}{\Delta t}( \partial_3 g\{y_\ast\}(t)
\cdot \int_b^{\sigma(t)} \partial_{4}F [y_\ast](\tau) \Delta \tau)\Bigr)
\end{split}
\end{equation*}
for all $t \in [a,b]^\kappa$, proving that $H= L-\lambda F$
satisfies the Euler--Lagrange equation (\ref{E-L equation for H}).
\end{proof}

\begin{Theorem}[Necessary optimality condition for normal and abnormal extremizers
of (\ref{functional 1})--(\ref{isoperimetric 1})]
\label{thm:abn}
Suppose that $y_\ast \in C_{rd}^1(\mathbb{T},\mathbb R)$ gives a local minimum or a local maximum
to the functional $\mathcal{L}$ subject to the boundary conditions (\ref{boundary conditions 1})
and the integral constraint (\ref{isoperimetric 1}). Then there exist two constants $\lambda_0$
and  $\lambda$, not both zero, such that $y_\ast$ satisfies the equation
\begin{equation}
\label{E-L equation for H normal and abnormal}
\displaystyle \partial_{2}H  [y](t)
- \frac{\Delta}{\Delta t}\partial_{3} H [y](t)
+\displaystyle \partial_2 g\{y\}(t) \cdot \int_{\sigma(t)}^b \partial_4 H [y](\tau)\Delta\tau
- \displaystyle \frac{\Delta}{\Delta t}\left(\partial_3 g\{y\}(t)
\cdot \int_{\sigma(t)}^b \partial_4 H [y](\tau) \Delta\tau\right)=0
\end{equation}
for all $t \in [a,b]^\kappa$, where $H = \lambda_0 L-\lambda F$.
\end{Theorem}

\begin{proof}
Following the proof of Theorem \ref{normalcase},
since $(0,0)$ is an extremizer of $\phi$ subject
to the constraint $\psi=0$, the
abnormal Lagrange multiplier rule
(\textrm{cf.}, \textrm{e.g.}, \cite{vanBrunt})
guarantees the existence of two reals
$\lambda_0$ and $\lambda$, not both zero, such that
$$
\lambda_0 \nabla \phi = \lambda \nabla \psi.
$$
Therefore,
$$
\lambda_0 \displaystyle\frac{\partial \phi}{\partial \epsilon_1}(0,0)
= \lambda \displaystyle\frac{\partial \psi}{\partial \epsilon_1}(0,0)
$$
and hence,
\begin{equation*}
\begin{split}
0 &= \displaystyle \int_a^b \Bigl( \lambda_0( \partial_{2}L[y_\ast](t)
- \frac{\Delta}{\Delta t}\partial_{3}L[y_\ast](t)-
\displaystyle \partial_2 g\{y_\ast\}(t) \cdot
\int_b^{\sigma(t)} \partial_{4}L [y_\ast](\tau) \Delta \tau \\
&\qquad + \displaystyle \frac{\Delta}{\Delta t}( \partial_3 g\{y_\ast\}(t)
\cdot \int_b^{\sigma(t)} \partial_{4}L [y_\ast](\tau) \Delta \tau ) )\\
&\qquad -  \displaystyle \lambda \Bigl(\partial_{2}F[y_\ast](t)
- \frac{\Delta}{\Delta t}\partial_{3}F[y_\ast](t)-
\displaystyle \partial_2 g\{y_\ast\}(t) \cdot
\int_b^{\sigma(t)} \partial_{4}F [y_\ast](\tau) \Delta \tau \\
&\qquad + \displaystyle \frac{\Delta}{\Delta t}( \partial_3 g\{y_\ast\}(t) \cdot
\int_b^{\sigma(t)} \partial_{4}F [y_\ast](\tau) \Delta \tau )\Bigr)\Bigr) \eta_1^\sigma (t) \Delta t.
\end{split}
\end{equation*}
From the arbitrariness of $\eta_1$ and Lemma~\ref{Fund. Lemma}
it is clear that equation
(\ref{E-L equation for H normal and abnormal}) holds for all
$t \in [a,b]^\kappa$, where $H = \lambda_0 L-\lambda F$.
\end{proof}

\begin{Remark}
Note that
\begin{enumerate}
\item If $y_\ast$ is a normal extremizer, then
one can consider, by Theorem~\ref{normalcase},
$\lambda_0 = 1$ in Theorem~\ref{thm:abn}.
The condition $(\lambda_0,\lambda) \ne (0,0)$ guarantees
that Theorem~\ref{thm:abn} is a useful necessary condition.
\item Theorem 3.4 of \cite{FerreiraTorres} is a corollary
of our Theorem~\ref{normalcase}:
in that case, $\partial_4 H=0$ and we simply obtain
$$
\partial_{2}H (t,y^\sigma(t),y^\Delta(t))
- \frac{\Delta}{\Delta t}\partial_{3} H(t,y^\sigma(t),y^\Delta(t)) =0
$$
for all $t \in [a,b]^\kappa$.
\end{enumerate}
\end{Remark}

We present two important corollaries that are obtained
from Theorem~\ref{thm:abn} choosing the time scale to be
$\mathbb{T}=h\mathbb{Z}:=\{ hz: z \in
\mathbb{Z}\}$, $h>0$, and $\mathbb{T}=q^{\mathbb{N}_0} :=\{ q^k: k \in
\mathbb{N}_0\}$, $q>1$.
In what follows we use the standard notation of \emph{quantum calculus}
(see, \textrm{e.g.}, \cite{KacCheung,MalinowskaTorres,withMiguel01}):
$$
\Delta_h y(t):=\displaystyle\frac{y(t+h)-y(t)}{h}
\quad \mbox{and} \quad
D_q y(t):=\displaystyle\frac{y(qt)-y(t)}{(q-1)t}.
$$

\begin{Cor}
Let $h>0$ and suppose that  $y_\ast$ is a solution to the discrete-time problem
$$
\mathcal{L}(y) = \sum_{t=a}^{b-h} L\left(t,y(t+h), \Delta_h y(t), z(t) \right)
\longrightarrow \mathrm{extr}
$$
with
$$
z(t)= \sum_{\tau=a}^{t-h} g\left(\tau,y(\tau+h), \Delta_h y(\tau)\right)
$$
in the class of functions $y$ satisfying the boundary conditions
$$
y(a)=\alpha \quad \mbox{ and } \quad y(b)=\beta
$$
and the constraint
$$
\mathcal{J}(y)
= \sum_{t=a}^{b-h} F\left(t,y(t+h), \Delta_h y(t), z(t) \right)
= \gamma
$$
for some given  $\alpha, \beta, \gamma \in \mathbb{R}$.
Then there exist two constants $\lambda_0$ and  $\lambda$,
not both zero, such that
\begin{equation*}
\begin{split}
0 &= \partial_{2}H(t,y_\ast(t+h),\Delta_h y_\ast(t), z_\ast(t))
- \Delta_h\partial_{3} H(t,y_\ast(t+h),\Delta_h y_\ast(t), z_\ast(t))\\
&\qquad + \displaystyle \partial_2 g(t,y_\ast(t+h),\Delta_h y_\ast(t))
\cdot \sum_{\tau=t+h}^{b-h} \partial_4 H(\tau,y_\ast(\tau+h),\Delta_h y_\ast(\tau), z_\ast(\tau))\\
&\qquad -  \displaystyle \Delta_h\left(\partial_3 g(t,y_\ast(t+h),\Delta_h y_\ast(t))
\cdot \sum_{\tau=t+h}^{b-h} \partial_4 H (\tau,y_\ast(\tau+h),\Delta_h y_\ast(\tau), z_\ast(\tau))\right)
\end{split}
\end{equation*}
for all $t \in \{a, a+h, \ldots, b-h\}$, where $H = \lambda_0 L-\lambda F$.
\end{Cor}

\begin{proof}
Choose $\mathbb{T}=h\mathbb{Z}$,
where $a,b \in \mathbb{T}$. The result follows
from Theorem~\ref{thm:abn}.
\end{proof}

\begin{Cor}
Let $q>1$ and suppose that $y_\ast$ is a solution to the quantum problem
$$
\mathcal{L}(y) =  \sum_{t=a}^{bq^{-1}}
L\left(t,y(qt), D_q y(t), z(t) \right)  \longrightarrow \mathrm{extr}
$$
with
$$
z(t)= \sum_{\tau=a}^{tq^{-1}} g\left(\tau,y(q\tau), D_q y(\tau)\right)
$$
in the class of functions $y$ satisfying the boundary conditions
$$
y(a)=\alpha \ \ \ \ \ \mbox{and} \ \ \ \ \ y(b)=\beta
$$
and the constraint
$$
\mathcal{J}(y) =  \sum_{t=a}^{bq^{-1}} F\left(t,y(qt), D_q y(t), z(t) \right)=\gamma
$$
for some given  $\alpha, \beta, \gamma \in \mathbb{R}$. Then
there exist two constants $\lambda_0$ and  $\lambda$, not both zero, such that
\begin{equation*}
\begin{split}
0&= \partial_{2}H(t,y_\ast(qt),D_q y_\ast(t), z_\ast(t))
- D_q\partial_{3} H(t,y_\ast(qt),D_q y_\ast(t), z_\ast(t))\\
&\qquad + \partial_2 g(t,y_\ast(qt),D_q y_\ast(t))
\cdot \sum_{\tau=qt}^{bq^{-1}} \partial_4 H(\tau,y_\ast(q\tau),D_q y_\ast(\tau), z_\ast(\tau))\\
&\qquad -  D_q\left(\partial_3 g(t,y_\ast(qt),D_q y_\ast(t))
\cdot \sum_{\tau=qt}^{bq^{-1}} \partial_4 H (\tau,y_\ast(q\tau),D_q y_\ast(\tau), z_\ast(\tau))\right)
\end{split}
\end{equation*}
for all $t \in \{a, q a, \ldots, bq^{-1}\}$,
where $H = \lambda_0 L-\lambda F$.
\end{Cor}

\begin{proof}
Choose $\mathbb{T}=q^{\mathbb{N}_0}$,
where $a,b \in \mathbb{T}$. The result follows
from Theorem~\ref{thm:abn}.
\end{proof}

%--------------------------------------------------

\subsection{Duality}
\label{duality}

In the paper \cite{Caputo} (see also \cite{Pawluszewicz:Torres,MyID:174})
Caputo states that the \emph{delta calculus} and the \emph{nabla calculus}
on time scales are the ``dual'' of each other.
A \emph{Duality Principle} is presented, that basically asserts that
it is possible  to obtain results
for the \emph{nabla calculus} directly
from results on the \emph{delta calculus} and vice versa.
Using the duality arguments of Caputo it is possible to prove easily
the \emph{nabla versions} of Theorem~\ref{E-L},
Theorem~\ref{Theorem natural boundary conditions},
Theorem~\ref{normalcase} and  Theorem~\ref{thm:abn}.

In what follows we assume that there exist at least three points
on the time scale: $r,a,b \in \mathbb{T}$ with $r<a<b$, and
that the operator $\rho$ is nabla differentiable. 
The following theorem is the
\emph{nabla version} of Theorem~\ref{thm:abn}, 
where the variational problem consists
of minimizing or maximizing the functional
\begin{equation}
\label{nabla-functional 1}
\mathcal{L}(y) =  \int_{a}^{b} L\left(
t,y^{\rho}(t), y^\nabla(t), z(t) \right) \nabla t,
\end{equation}
the variable $z$ in the integrand being itself
expressed in terms of a nabla indefinite integral
$$
z(t)=\int_{a}^{t} g\left(
\tau,y^{\rho}(\tau), y^\nabla(\tau) \right) \nabla \tau,
$$
in the class of functions $y \in C_{ld}^1(\mathbb{T},\mathbb R)$
satisfying the boundary conditions
\begin{equation}
\label{nabla-boundary conditions 1}
y(a)=\alpha \quad \mbox{ and } \quad y(b)=\beta
\end{equation}
and the nabla integral constraint
\begin{equation}
\label{nabla-isoperimetric 1}
\mathcal{J}(y)
= \int_{a}^{b} F\left(t,y^{\rho}(t), y^\nabla(t), z(t) \right) \nabla t
=\gamma
\end{equation}
for some given  $\alpha, \beta, \gamma \in \mathbb{R}$. We assume that

\begin{enumerate}
\item the admissible functions $y$ belong to the class $C^{1}_{ld}(\mathbb{T},\mathbb{R})$;
\item $(t,y,v,z) \rightarrow L(t,y,v,z)$ \ and  \ $(t,y,v,z) \rightarrow F(t,y,v,z)$
have continuous partial derivatives with respect to $y,v,z$ for all $t \in [a,b]$;
\item $(t,y,v) \rightarrow g(t,y,v)$ has continuous partial
derivatives with respect to $y,v$ for all $t \in [a,b]$;
\item $t \rightarrow L(t, y^\rho(t), y^\nabla(t),z(t))$
and $t \rightarrow F(t, y^\rho(t), y^\nabla(t),z(t))$ belong to the class
$C_{ld}(\mathbb{T},\mathbb{R})$ for any admissible function $y$;
\item $t \rightarrow  \partial_3 L(t, y^\rho(t), y^\nabla(t),z(t))$,
$t \rightarrow  \partial_3 F(t, y^\rho(t), y^\nabla(t),z(t))$
and $t \rightarrow \partial_3 g(t, y^\rho(t), y^\nabla(t))$
belong to the class $C^{1}_{ld}(\mathbb{T},\mathbb{R})$
for any admissible function $y$.
\end{enumerate}

The following operators are used:
$$
\lceil y\rfloor  (t):=(t,y^\rho(t),y^\nabla(t),z(t))
\quad \mbox{and} \quad \langle y
\rangle(t):=(t,y^\rho(t),y^\nabla(t)),
\quad \mbox{ where } \quad y \in C^1_{ld}(\mathbb{T},\mathbb{R}).
$$

\begin{Theorem}[Necessary optimality condition for normal and abnormal extremizers
of (\ref{nabla-functional 1})--(\ref{nabla-isoperimetric 1})]
\label{nabla-thm:abn}
Suppose that $y_\ast \in C_{ld}^1(\mathbb{T},\mathbb R)$
gives a local minimum or a local maximum to the functional $\mathcal{L}$
subject to the boundary conditions (\ref{nabla-boundary conditions 1})
and the integral constraint (\ref{nabla-isoperimetric 1}). Then there exist
two constants $\lambda_0$ and  $\lambda$, not both zero, such that
$y_\ast$ satisfies the equation
$$
\displaystyle \partial_{2}H  \lceil y \rfloor (t)
- \frac{\nabla}{\nabla t}\partial_{3} H \lceil y \rfloor  (t)
+\displaystyle \partial_2 g \langle y \rangle (t)
\cdot \int_{\rho(t)}^b \partial_4 H \lceil y \rfloor  (\tau)\nabla\tau
- \displaystyle \frac{\nabla}{\nabla t}\left(\partial_3 g \langle y \rangle(t)
\cdot \int_{\rho(t)}^b \partial_4 H \lceil y \rfloor (\tau) \nabla\tau\right)=0
$$
for all $t \in [a,b]_\kappa$, where $H = \lambda_0 L-\lambda F$.
\end{Theorem}

\begin{Remark}
Theorem~2 of \cite{Almeida:Torres} is a corollary
of our Theorem~\ref{nabla-thm:abn}:
in that case $\partial_4 H=0$, and one obtains
$$
\partial_{2}H (t,y^\rho(t),y^\nabla(t))
- \frac{\nabla}{\nabla t}\partial_{3} H(t,y^\rho(t),y^\nabla(t)) =0
$$
for all $t \in [a,b]_\kappa$.
\end{Remark}

From Theorem~\ref{E-L}, via duality, one can easily obtain the Euler--Lagrange
equation for the nabla problem (\ref{nabla-functional 1})--(\ref{nabla-boundary conditions 1})
(or from Theorem~\ref{nabla-thm:abn} noting that, since there is no nabla integral constraint,
$F=0$ and $\gamma=0$).

\begin{Theorem}[Necessary optimality condition
to (\ref{nabla-functional 1})--(\ref{nabla-boundary conditions 1})]
\label{nabla-E-L}
Suppose that $y_{\ast}$ is a local minimizer or local maximizer
to problem (\ref{nabla-functional 1})--(\ref{nabla-boundary conditions 1}).
Then $y_{\ast}$ satisfies the Euler--Lagrange equation
\begin{equation}
\label{nablaELequation}
\displaystyle \partial_{2}L \lceil y \rfloor(t)
- \frac{\nabla}{\nabla t}\partial_{3} L \lceil y \rfloor(t)
+ \displaystyle \partial_2 g\langle y \rangle(t)
\cdot \int_{\rho(t)}^b \partial_4 L \lceil y \rfloor(\tau)\nabla\tau
- \displaystyle \frac{\nabla}{\nabla t}\left(\partial_3 g\langle y \rangle(t)
\cdot \int_{\rho(t)}^b \partial_4 L \lceil y \rfloor(\tau)\nabla\tau\right)=0
\end{equation}
for all $t\in\left[a,b\right]_\kappa$.
\end{Theorem}

\begin{Remark}
As a corollary of Theorem~\ref{nabla-E-L}
we obtain the Euler--Lagrange equation for the basic problem
of the calculus of variations on nabla calculus
\cite{MartinsTorres} (see also \cite{Atici}).
In  that case $\partial L_4=0$ and one obtains that
$$
\partial_{2}L(t,y^{\rho}(t),y^\nabla(t))
- \frac{\nabla}{\nabla t}\partial_{3}L(t,y^{\rho}(t),y^\nabla(t))=0
$$
for all $t\in\left[  a,b\right]_\kappa$.
\end{Remark}

\begin{Remark} Theorem~\ref{nabla-E-L} gives the Euler--Lagrange equation in the
nabla--differential form. The
Euler--Lagrange equation in the nabla-integral form to
problem (\ref{nabla-functional 1})--(\ref{nabla-boundary conditions 1}) is
$$
\displaystyle \partial_{3}L \lceil y \rfloor(t)
+ \partial_3 g\langle y \rangle(t)
\cdot \int_{\rho(t)}^b \partial_4 L \lceil y \rfloor(\tau)\Delta\tau
+ \int_t^b \Big(\partial_2 L \lceil y \rfloor(s)+
\displaystyle \partial_2 g\langle y \rangle(s)
\cdot \int_{\rho(s)}^b \partial_{4} L \lceil y \rfloor(\tau) \Delta\tau\Big)\Delta s=const.
$$
\end{Remark}

Applying the duality arguments of Caputo
to Theorem~\ref{Theorem natural boundary conditions}
the following result is obtained.

\begin{Theorem}[Natural boundary conditions to (\ref{nabla-functional 1})]
Suppose that $y_{\ast}$ is  a local minimizer (resp. local maximizer) to problem
(\ref{nabla-functional 1}).
Then $y_{\ast}$ satisfies
the Euler--Lagrange equation
(\ref{nablaELequation}). Moreover,
\begin{enumerate}

\item if $y(a)$ is free, then the natural boundary condition
\begin{equation*}
\partial_3 L\lceil y_\ast\rfloor(a)= - \partial_3 g\langle y_\ast\rangle(a)
 \cdot \int_{\rho(a)}^b \partial_{4}L\lceil y_\ast\rfloor(\tau)\Delta\tau
\end{equation*}
holds;
\item if $y(b)$ is free, then the natural boundary condition
\begin{equation*}
\partial_3 L\lceil y_\ast\rfloor(b)=- \partial_3 g\langle y_\ast\rangle(b)
\cdot \int_{\rho(b)}^b \partial_{4}L\lceil y_\ast\rfloor(\tau)\Delta\tau
\end{equation*}
holds.
\end{enumerate}
\end{Theorem}

%--------------------------------------------------

\section{Applications}
\label{subsection applications}

From now on we assume that $\mathbb{T}$ satisfies the following condition $(H)$:
\begin{description}
\item[$(H)$] \qquad for each $t \in\mathbb{T}$, $\rho(t)= a_1 t +a_0$ for some
$a_1\in\mathbb{R}^+$ and $a_0\in\mathbb{R}$.
\end{description}

\begin{Remark}
Note that condition $(H)$ implies that $\rho$ is nabla
differentiable and
$\rho^{\nabla}(t)=a_1$, $t \in \mathbb{T}_\kappa$. Also note that
condition $(H)$ englobes the differential calculus
($\mathbb{T}=\mathbb{R}$, $a_1=1$, $a_0=0$), the difference
calculus ($\mathbb{T}=\mathbb{Z}$, $a_1=1$, $a_0=-1$),
the $h$-calculus ($\mathbb{T}=h \mathbb{Z}$, for some $h>0$, $a_1=1$, $a_0=-h$),
and the
q-calculus ($\mathbb{T}=q^{\mathbb{N}_0}$ for some $q>1$, $a_1=\frac{1}{q}$, $a_0=0$).
\end{Remark}

The following result illustrates an application of Theorem~\ref{nabla-thm:abn}.

\begin{Proposition}
\label{Propo 1}
Suppose that $\mathbb{T}$ satisfies condition $(H)$,
$\xi$ is a real parameter, and $k \in \mathbb{R}$ is a given constant.
Suppose that $f:\mathbb{R}^2 \rightarrow\mathbb{R}$
is a $C^2$ function that satisfies the conditions:
\begin{enumerate}
\item[(A1)] $\partial_1 f(y^\rho (t), \xi)\neq - k a_1$ for all $t$
in some non-degenerate interval $I \subseteq [a,b]$,
for all $\xi$ and for all admissible function $y$;

\item[(A2)] $\partial^{2}_{1,1} f(y^\rho (t), \xi)\neq 0$ for all $t$
in some non-degenerate interval $I \subseteq [a,b]$,
for all $\xi$ and for all admissible function $y$.
\end{enumerate}
Consider
$$
L(t,y,v,z)=f(y,\xi)+kz, \ \ g(t,y,v)=v \ \  \mbox{and}  \ \ F(t,y,v,z)=y.
$$
If $y_\ast$ is a solution to problem
(\ref{nabla-functional 1})--(\ref{nabla-isoperimetric 1}),
then $y_\ast(t)=\alpha$, $t\in [a,b]^\kappa$.
\end{Proposition}

\begin{proof}
Suppose that $y_\ast$ is an extremizer to problem
(\ref{nabla-functional 1})--(\ref{nabla-isoperimetric 1}).
By Theorem~\ref{nabla-thm:abn} there exist two constants $\lambda_0$ and  $\lambda$,
not both zero, such that $y_\ast$ satisfies the equation
\begin{equation}
\label{nabla-abn}
\displaystyle \partial_{2}H  \lceil y \rfloor (t)
- \frac{\nabla}{\nabla t}\partial_{3} H \lceil y \rfloor  (t)
+\displaystyle \partial_2 g \langle y \rangle (t) \cdot
\int_{\rho(t)}^b \partial_4 H \lceil y \rfloor  (\tau)\nabla\tau
- \displaystyle \frac{\nabla}{\nabla t}\left(\partial_3 g \langle y \rangle(t)
\cdot \int_{\rho(t)}^b \partial_4 H \lceil y \rfloor (\tau) \nabla\tau\right)=0
\end{equation}
for all $t \in [a,b]_\kappa$, where $H = \lambda_0 L-\lambda F$.
Since
$$
\partial_2 H=\lambda_0 \partial_1 f - \lambda, \ \ \partial_3 H=0,
\ \ \partial_4 H=\lambda_0 k, \ \ \partial_2 g=0 \ \ \mbox{and} \ \ \partial_3 g=1,
$$
then equation (\ref{nabla-abn}) reduces to
\begin{equation}
\label{nabla-abn1}
\lambda_0 \Bigl(\partial_1 f(y_\ast^\rho (t), \xi)
+ k a_1 \Bigr) = \lambda , \ \ \ t \in [a,b]_\kappa.
\end{equation}
Note that if $\lambda_0=0$, then $\lambda=0$ violates the condition that
$\lambda_0$ and $\lambda$ do not vanish simultaneously. If $\lambda=0$,
then equation (\ref{nabla-abn1}) reduces to
$\lambda_0 \Bigl(\partial_1 f(y_\ast^\rho (t), \xi) + k a_1 \Bigr)=0$.
By assumption $(A1)$ we conclude that $\lambda_0=0$, which again contradicts
the fact that $\lambda_0$ and $\lambda$ are not both zero. Consequently,
we can assume, without loss of generality, that $\lambda_0=1$.
Hence, equation (\ref{nabla-abn1}) takes the form
$$
\partial_1 f(y_{\ast}^\rho (t), \xi)=\lambda-k a_1,
\quad t \in [a,b]_\kappa.
$$
By assumption $(A2)$ we conclude that
$$
y_\ast^\rho(t)=const, \quad t\in [a,b]_\kappa.
$$
Since $y(a)=\alpha$, we obtain that $y_\ast(t)=\alpha$
for any $t\in [a,b]^\kappa$.
\end{proof}

Observe that the solution to the class of problems 
considered in Proposition~\ref{Propo 1} is a
constant function that depends only on the boundary
conditions (and the isoperimetric constraint) but not
explicitly on the integrand function and its parameters.

\begin{Remark}
By the isoperimetric constraint (\ref{nabla-isoperimetric 1}),
a necessary condition for the problem of Proposition~\ref{Propo 1}
to have a solution is that  $\alpha=\displaystyle\frac{\gamma}{b-a}$.
\end{Remark}

\begin{Remark}
Let $b$ be a left dense point. Then, by the boundary conditions
(\ref{nabla-boundary conditions 1}), a necessary condition
for the problem of Proposition~\ref{Propo 1}
to have solution is that $\alpha=\beta$.
\end{Remark}

\begin{Remark}[\textrm{cf.} \cite{Caputo2009}]
Let $\mathbb{T}=\mathbb{R}$. Suppose that
$\alpha=\displaystyle\frac{\gamma}{b-a}=\beta$.
\begin{enumerate}
\item If $\partial^{2}_{1,1} f(y (t), \xi)>0$ for all $t \in [a,b]$,
for all $\xi$ and for all admissible function $y$, then problem
(\ref{nabla-functional 1})--(\ref{nabla-isoperimetric 1})
has a unique minimizer.
\item If $\partial^{2}_{1,1} f(y (t), \xi)<0$ for all $t \in [a,b]$,
for all $\xi$ and for all admissible function $y$, then problem
(\ref{nabla-functional 1})--(\ref{nabla-isoperimetric 1})
has a unique maximizer.
\end{enumerate}
\end{Remark}

We end the paper with an example of application
of the \emph{nabla version} of Theorem~\ref{normalcase}.

\begin{Ex}
Let $q:[a,b] \rightarrow \mathbb{R}$ be a continuous function
and $y^{\nabla^2}:= (y^{\nabla})^{\nabla}$.
Suppose that $y_\ast \in C^2_{ld}$ is an extremizer for
\begin{equation*}
\mathcal{L}(y) = \int_{a}^{b} \Bigl((y^\nabla)^2(t)-q(t)(y^\rho)^2(t)
+ 2\int_a^t y^\nabla (\tau)\nabla\tau \Bigr)\nabla t
\end{equation*}
subject to the boundary conditions
\begin{equation*}
y(a)=0 \quad \mbox{ and } \quad y(b)=0
\end{equation*}
and the delta integral constraint
\begin{equation}
\label{isoperimetric 1Ex}
\mathcal{J}(y) = \int_{a}^{b}(y^\rho)^2(t)\nabla t  =1.
\end{equation}
Note that any extremal to $\mathcal{J}$ does not satisfy the isoperimetric constraint
(\ref{isoperimetric 1Ex}). Hence, this problem has no abnormal extremizers and,
by the nabla version of Theorem~\ref{normalcase}, there exists $\lambda \in \mathbb{R}$
such that $y_\ast$ satisfies the equation
\begin{equation}
\label{eq.Ex}
\displaystyle \partial_{2}H \lceil y \rfloor(t)
- \frac{\nabla}{\nabla t}\partial_{3} H \lceil y \rfloor(t)
+\displaystyle \partial_2 g\langle y\rangle(t) \cdot \int_{\rho(t)}^b \partial_4 H \lceil y \rfloor(\tau)\nabla\tau
- \displaystyle \frac{\nabla}{\nabla t}\left(\partial_3 g \langle y \rangle(t)
\cdot \int_{\rho(t)}^b \partial_4 H \lceil y \rfloor(\tau) \nabla\tau\right)=0
\end{equation}
for all $t \in [a,b]_\kappa$, where $H = L-\lambda F$ and
$$
L(t,y,v,z)= v^2 - q(t) y^2 + 2z, \quad
g(t,y,v)=v, \quad  \mbox{ and }  \quad F(t,y,v,z)=y^2.
$$
Since
$$
\partial_2 H=-2qy - 2\lambda y, \ \ \partial_3 H=2v, \ \ \partial_4 H=2,
\ \ \partial_2 g=0, \ \ \mbox{and} \ \ \partial_3 g=1,
$$
then equation (\ref{eq.Ex}) reduces to
\begin{equation}
\label{Sturm-eq}
y^{\nabla^2}(t)+q(t)y^\rho(t)+\lambda y^\rho (t)
= \partial_4 H \cdot \frac{a_1}{2}, \ \ t \in [a,b]_{\kappa^2}.
\end{equation}
Note that in the basic problem of calculus of variations on time scales,
$\partial_4 H=0$, and we obtain the nabla version of the well known Sturm--Liouville eigenvalue equation:
$$
y^{\nabla^2}(t)+q(t)y^\rho(t)+\lambda y^\rho (t)=0, \ \ t \in [a,b]_{\kappa^2}
$$
(see \cite{Agarwal:Bohner,FerreiraTorres}). The study of solutions to equation
(\ref{Sturm-eq}) in the case $\partial_4 H \ne 0$ is an interesting open problem.
\end{Ex}

%--------------------------------------------------

\section*{Acknowledgments}

The authors are grateful to the support
of the \emph{Portuguese Foundation for Science and Technology} (FCT)
through the \emph{Center for Research and Development
in Mathematics and Applications} (CIDMA).

%--------------------------------------------------

%--------------------------------------------------

\end{document}